\documentclass[12pt,reqno]{amsart}
\newtheorem{theorem}{Theorem}[section]   
\newtheorem{proposition}[theorem]{Proposition}  
\newenvironment{remark}{{\bf Remark.\ }\rm}{\bigskip}  
\newenvironment{example}{\bf Example. \rm}{\bigskip}
\newenvironment{definition}{\bf Definition.\ \rm}{\bigskip}
\newenvironment{notation}{{\bf Notation.\ }\rm}{\bigskip}
  
\renewenvironment{abstract}{\begin{quote}{\bf Abstract.\ }\small}{\end{quote}\bigskip}
\def\arcsinh{{\rm arcsinh}\,}
\def\sech{{\rm sech}\,}
\def\arctanh{{\rm arctanh}\,}
\title{Canonical Polynomial Sequences: Inverse Pairs}

\author{Philip Feinsilver}
\address{Department of Mathematics\\
Southern Illinois University\\
Carbondale, IL. 62901, U.S.A.}

\begin{document}

\maketitle

\thispagestyle{empty}

\begin{abstract}    
For $V(z)$, analytic in a neighborhood of $0\in\mathbb{C}$, $V(0) = 0$, $V'(0)\ne0$,
there is an associated sequence of polynomials, \textsl{canonical polynomials},
that is a generalized Appell sequence with lowering operator $V(d/dx)$.  
Correspondingly, the inverse function $Z(v)$ has an associated canonical polynomial sequence.  
The coefficients of the two sequences form two mutually inverse infinite matrices.   
We  detail  the  operator  calculus  for  the  pair  of  systems,  illustrating  the  duality  of  operators  and  variables between them.  
Examples are presented including a connection between Gegenbauer and Bessel polynomials. Touchard polynomials appear as well.  
Alternative settings would include an umbral calculus approach as well as the Riordan group.
Here we use the approach through operator calculus.
\end{abstract}

{\bf Keywords: }{Riordan group, special functions, orthogonal polynomials, operator calculus, Heisenberg-Weyl algebra, Bessel operator}\\

{\bf Subject Classification: }{Primary 33C45, 05A40; Secondary 42C05, 17B30}

\begin{section}{Introduction}

Our goal is to study certain families of polynomials, \textsl{canonical polynomial systems}, which
are a type of generalized Appell systems. The term \textsl{canonical} comes from the phrase
``canonical commutation relations" referring, in our case, to operators acting on polynomials satisfying
the boson commutation relation(s), i.e., operators whose commutator is the identity. We will present these in explicit form.\\

Acting on polynomials in $x$, define the operators
$$D=\frac{d}{dx} \qquad \text{and}\qquad x=\hbox{``multiplication by }x \hbox{"}$$ 
the symbol $x$ playing the dual r\^ole of variable and multiplication operator.\\

They satisfy commutation relations $[D,x]=I$, where $I$, the identity
operator, commutes with both $D$ and $x$. Abstractly, the Heisenberg-Weyl algebra is the associative
algebra generated by operators $\{A,B,C\}$ satisfying the commutation relations
$$[A,B]=C\,,\quad [A,C]=[B,C]=0\;.$$
The {\sl standard\/} HW algebra is the one generated by the realization $\{D,x,I\}$. So in this work, we are
studying particular representations of the HW algebra acting on spaces of polynomials.\\

A system of polynomials $\{p_n(x)\}_{n\ge0}$ is an {\sl Appell system\/} if it is a basis for 
a representation of the standard HW algebra with the following properties:
\begin{enumerate}  
\item $p_n$ is of degree $n$ in $x$
\item $D\,p_n=n\,p_{n-1}$.
\end{enumerate}

Such sequences of polynomials are canonical polynomial systems in the sense
that they provide a polynomial basis for a representation of the 
Heisenberg-Weyl algebra, in realizations different from the
standard one. Our approach is to  use operator calculus methods for studying such systems.\\

More specifically, take an analytic function $V(z)$ defined in a neighborhood of 0 in ${\textbf C}$,
with $V(0)=0$, $V'(0) \ne 0$. Denote $W(z)=1/V'(z)$ and $Z(v)$ the inverse function, i.e.,
$V(Z(v))=v$, $Z(V(z))=z$. \\

Now, $V(D)$ is defined by power series as an operator on polynomials
in $x$. We have the commutation relations (see below, Proposition \ref{prop:CR})
$$[V(D),X]=V'(D)\,,\qquad [V(D),xW(D)]=I\;.$$ 
In other words, $V=V(D)$ and $Y=xW(D)$ 
generate a representation of the HW algebra on polynomials in $x$. The basis for the
representation is $p_n(x)=Y^n1$. The constant function $1=p_0$ is referred to as a ``vacuum state", as
it is annihilated by the lowering operator $V$: $V(D)1=V(0)=0$. \\

We generate the polynomial basis recursively:
\begin{equation}\label{eq:rec1}
 p_{n+1} =  Yp_n\,,\qquad V(D)p_n=n\,p_{n-1}
\end{equation}

That is, $Y$ is the {\sl raising operator\/} and $V(D)$ is the corresponding {\sl lowering operator\/}.  
Thus, $\{p_n\}_{n\ge0}$ form a system of {\sl canonical polynomials\/}.
The operator of multiplication by $x$ is given by 
$$x=YV'(D)=YZ'(V)^{-1}$$ 
which is a {\sl recursion operator\/} for the system. \\

Our approach holds as well for multivariate systems, some examples of which will appear in the
final sections of this work.\\

Primarily this work is a source providing useful examples of such sequences with
details about the operators/representations involved. Many well-known sequences of polynomials, or closely-related
families, will make their appearance.
\end{section}

\begin{section}{Dual vector fields and canonical polynomial sequences}

Identifying vector fields with first-order partial differential operators,
consider a variable $A$ with corresponding partial differential operator $\partial_A$.
Given $V$ as above, let $\tilde Y$ be the vector field $\tilde Y=W(A)\,\partial_A$.
An operator function of $D$ acts on $e^{Ax}$
as multiplication by the corresponding function in the variable $A$.\\

The basic fact is the following primitive version of the Fourier transform:
\begin{equation}\label{eq:FT}
 x W(D)\, e^{A x}=W(A) \partial_A\, e^{A x}=x W(A)\, e^{A x}
\end{equation}
where $D$ denotes $d/dx$, and $W$ is a function analytic in a neighborhood of $0\in\mathbb{C}$.\\

What this does is exchange the operators in the $(x,D)$ variables with corresponding
operators in $(A,\partial_A)$ variables. Thus each vector field has its dual and vice versa. That is,
$xW(D)$ is the \textsl{dual vector field} corresponding to the vector field $W(A)\partial_A$.
We abbreviate ``dual vector field" by ``dvf".\\

When discussing functions of $D$, we use the variable $z$ for complex
variables defining the function involved. For an operator $f(D)$, the function $f(z)$ may be
referred to as its \textsl{symbol}.
We have the commutation relation, $Dx-xD=I$, $x$ playing the dual r\^ole of variable
and operator of multiplication by $x$. When there is no ambiguity of meaning, we write as well $[z,x]=I$.\\

We have the ``canonical function" $V(z)$ holomorphic in a neighborhood of $0 \in \mathbb{C}$, with $V(0)=0$, and the derivative $V'(0)\ne0$.
And the functional inverse of $V$, denoted by $Z(v)$.\\

The associated dvf is called a ``canonical variable" and is given by
$$Y=xW(D)$$ 
where $W(z)=V'(z)^{-1}$ is the reciprocal of the derivative $V'(z)$. \\

We remark the facilitating feature of this operator calculus that any function of $D$ expressible as a series expansion, perhaps formal,
acts finitely on polynomials. Defining the action on exponentials as evaluation gives a well-defined calculus acting on polynomials, exponentials, and
linear combinations of exponentials with polynomial coefficients. Cf. \cite{FS,RKO}.

\begin{proposition} \label{prop:CR} The commutation relation
$$[V(D),x]=V'(D)$$
the derivative of $V$.
\end{proposition}
\begin{proof} We use the approach of equation\eqref{eq:FT} and the usual product rule:
\begin{align*}
V(D)x\,e^{Ax}&=V(D)\partial_A\,e^{Ax}=\partial_A V(A)\,e^{Ax}\\ 
&=V'(A)\,e^{Ax}+V(A)x\,e^{Ax}=V'(D)\,e^{Ax}+xV(D)\,e^{Ax}
\end{align*}
so that
$$V(D)x-xV(D)=V'(D)$$
acting on $e^{Ax}$. Then differentiating repeatedly with respect to $A$ yields the result for polynomials as well.
\end{proof}

\begin{proposition}
The commutation relations
$$[V(D),Y]=I$$
hold.
\end{proposition}
\begin{proof}
By the above Proposition:
$$[V(D),xW(D)]=V'(D)W(D)=I$$
as required.
\end{proof}

\begin{notation} 
We denote the operator $V(D)$ by $\mathcal{V}$.
\end{notation}

 In the $(Y,\mathcal{V})$ calculus, parallel to the usual $(x,D)$ calculus, $Y$ takes the r\^ole of multiplication by $x$ and $\mathcal{V}$ the r\^ole of differentiation.\\

\begin{remark}
 Note that $W$ is essentially the derivative of the inverse function $Z(v)$. That is,
$$ Z(V(z))=z \Longrightarrow Z'(V(z))V'(z)=1 \Longrightarrow Z'(V(z))=W(z)$$
and substituting $z=Z(v)$,
$$W(Z(v))=Z'(v)$$
relations well-known from calculus.
\end{remark}

\begin{example}  For the \textsl{Poisson case}, we take $V(z)=e^z-1$, so that $W(z)=e^{-z}$ and
$Z(v)=\log(1+v)$. 
\end{example}

\begin{subsection}{Canonical polynomials}
Given $V$, we define the corresponding \textsl{canonical polynomials}, $\{p_n(x)\}$ by
$$p_n(x)=(xW(D))\cdots(xW(D))1=Y^n1$$
so that $p_n$ satisfy the relations
\begin{align*}
Yp_n&=p_{n+1}\\
\mathcal{V}p_n&=np_{n-1}
\end{align*}
with $p_0=1$, $\mathcal{V}p_0=0$. Because of this property, $\mathcal{V}$ is called the lowering operator or velocity operator.\\

\begin{definition}     
Given a canonical polynomial sequence we write
$$p_n(x)=\sum_k p_{nk} x^k$$
\end{definition}

We call the infinite matrix $\{p_{nk}\}$, the \textsl{matrix of coefficients} or simply refer to the coefficients of the polynomial sequence.
We can refer as well to the $P$-matrix of the system. \\

Observe (e.g., setting $x=0$ in the generating function) that $p_n(0)=0$ for all $n>0$.\\

\begin{remark} 
One can work directly with the coefficients $\{p_{nk}\}$, as infinite matrices, as basic objects of study. Cf. \cite{SGWW,V1,V2}.
\end{remark}

\begin{example}For the Poisson case, $W(z)=e^{-z}$ acts on functions of $x$ as a shift operator
$$W(D)\,f(x)=f(x-1)$$ 
so $Yf(x)=xf(x-1)$ and iterating, starting from $p_0=1$, yields
$$p_n(x)=x(x-1)\cdots(x-(n-1))=x^{(n)}$$
falling factorial polynomial. One checks directly that $\mathcal{V}=e^D-1$ acts accordingly.
\end{example}

We continue by looking at several basic examples and structures associated to sequences of canonical polynomials.
\end{subsection}

\begin{subsection}{Flow of a dual vector field}
We will now derive the exponential generating function for the sequence $\{p_n\}$. From the connection between vector fields and
dvf's, observe that we are interested in the flow of the dual vector field.\\

We would like to calculate the solution to
$$\frac{\partial u}{\partial t}=Yu,\qquad\qquad u(0)=f(x)$$
for the dvf $Y=x W(D)$.\\

Let $\tilde Y=W(A)\partial_A$ denote the corresponding vector field. Using equation \eqref{eq:FT}, we have
$$e^{tY}e^{A x}=e^{t\tilde Y}e^{A x}$$
by, say, applying $Y$ iteratively on the left side, and on the right side moving $Y$ past powers of $\tilde Y$ and converting to $\tilde Y$ when
acting on $e^{Ax}$. \\

We know how to calculate the flow of the vector field $\tilde Y$ by the characteristic equations
$$\dot A=W(A)$$
Multiplying both sides by $V'(A)$ we get
$$\dot A\,V'(A)=1$$
Now the left-hand side is an exact derivative. I.e.,
$$\frac{d}{dt}\,V(A(t))=1$$
Integrating, with initial conditions $A(0)=A$, we get
$$V(A(t))=V(A)+t$$
Solving, we have
$$A(t)=Z(t+V(A))$$

Thus,
\begin{equation}\label{eq:flow}
e^{tY}e^{A x}=e^{t\tilde Y}e^{A x}=e^{x Z(t+V(A))}
\end{equation}

As a corollary we have the action of the dvf $Y$ on the vacuum function equal
to 1 by setting $A=0$:
$$e^{tY}1=e^{x Z(t)}$$
In our context, replacing $t$ by $v$ we have

\begin{theorem}{\rm Main Formula}\\
For a canonical function $V(z)$, with $W(z)=1/V'(z)$, $Z(V(z))=z$, let $Y=xW(D)$ be the associated
canonical variable. Then we have
$$ e^{vY}e^{Ax}=e^{xZ(v+V(A))}$$
And, in particular,
$$ e^{vY}1=e^{xZ(v)}\ .$$
\end{theorem}

We have the expansion 
$$e^{v Y}1=e^{x Z(v)}=\sum_{n\ge0} \frac{v^n}{n!}\,p_n(x)$$

\begin{example}
For the Poisson case: $V(z)=e^z-1$, $Y=xe^{-D}$. We found above that
$$p_n(x)=Y^n1=x(x-1)\cdots(x-n+1)=x^{(n)}$$
With $Z(v)=\log(1+v)$, the expansion is
$$(1+v)^x=\sum_{n\ge 0}\frac{v^n}{n!} x^{(n)}=\sum_{n\ge0}  \binom{x}{n}\,v^n$$
the standard binomial theorem.
\end{example}
\end{subsection}
	
\begin{subsection}{Semigroup and group properties}
Observing that the generating function is an exponential, 
\begin{proposition}
We have the group property, \textsl{binomial type} relations, cf. \cite{R,RR}:
$$\sum_j\binom{n}{j} p_{n-j}(s)p_j(t)=p_n(s+t)\ .$$
\end{proposition}
which follows upon expanding
$$e^{sZ(v)}e^{tZ(v)}=e^{(s+t)Z(v)}$$
and regrouping. We have as well relations for the coefficients $\{p_{nk}\}$:
\begin{proposition}
The coefficients satisfy the semigroup property:
$$\sum_j\binom{n}{j}p_{n-j,k}\,p_{jl}=\binom{k+l}{k}\,p_{n,k+l}\ .$$
\end{proposition}
\begin{proof}
Expand the generating function two ways:
$$e^{xZ(v)}=\sum_k x^k\,\frac{1}{k!}\,Z(v)^k=\sum_n\frac{v^n}{n!}\,\sum_k p_{nk}x^k$$
and see that		
$$Z(v)^k=\,\sum_n \frac{v^n\,k!\,p_{nk}}{n!}$$
Now use the semigroup property $Z(v)^kZ(v)^l=Z(v)^{k+l}$ and rearrange factorials to arrive at the result.
\end{proof}
\end{subsection}
\begin{subsection}{Composition of systems. Riordan group. Umbral composition.}
\hfill\break
Note. For simplicity, the range of summations is suppressed in these calculations. Generically $n\ge0$ running to $\infty$ with $k$ running from 0 to $n$.\\

Start with two systems
\begin{align*}
e^{xZ_1(v)}&=\sum\frac{v^n}{n!}\sum p_{nk}^{(1)}x^k=e^{vY_1}1\\
e^{xZ_2(v)}&=\sum\frac{v^n}{n!}\sum p_{nk}^{(2)}x^k=e^{vY_2}1
\end{align*}
When composing these systems, we will introduce scaling factors for additional flexibility. Consider the composed
$P$-matrices and resum (with $\gamma$ and $\sigma$ as scale factors):
\begin{align*}
\sum \frac{v^n}{n!}\,\gamma^n\sum\sum p_{nm}^{(2)}\sigma^m\,p_{mk}^{(1)}x^k&=
\sum \frac{(v\gamma)^n}{n!}\,\sum p_{nm}^{(2)}\sigma^m\,Y_1^m1\\
&=e^{(\sigma Y_1) Z_2(\gamma v)}1
\end{align*}
with $\sigma Y_1$ playing the r\^ole of $x$ for the $Z_2$ system. Now consider $\sigma Z_2(\gamma v)$ as the $v$-variable for the $Z_1$ system:
$$e^{(\sigma Z_2(\gamma v))Y_1}1=\exp\bigl(x Z_1(\sigma Z_2(\gamma v)) \bigr)\ .$$
That is,
\begin{theorem}{\rm Composition Theorem}\\
We have the expansion
$$\exp\bigl(x Z_1(\sigma Z_2(\gamma v)) \bigr)=\sum \frac{v^n}{n!}\,p_n^{(12)}(x)$$
where
$$p_n^{(12)}(x)=\gamma^n\sum\sum p_{nm}^{(2)}\sigma^m\,p_{mk}^{(1)}x^k\ .$$
\end{theorem}
Call $V_{12}$ the canonical operator for the combined system. Then we can solve:
$$z=Z_1\bigl(\sigma Z_2(\gamma V_{12})\bigr) \Longrightarrow V_{12}=\frac{1}{\gamma}\,V_2\bigl(\frac{1}{\sigma}\,V_1(z)\bigr)$$
 \begin{remark}
1. The construction for polynomial $f$:
$$f(x)=\sum c_m x^m \rightarrow \sum c_m p_m(x)=f(Y)1$$ 
is \textsl{umbral composition}. Cf. \cite{R,RR}.\\

2. Composing the systems by working with the $P$-matrices directly is the technique of the \textsl{Riordan group}, \cite{SGWW}.\\

Our approach  uses operators and variables explicitly. 
\end{remark}


\begin{subsubsection}{Poisson subordination}	\label{sec:PS}
A common situation arises where \linebreak $Z_2(v)=e^v-1$. Then $p_{nm}^{(2)}=S(n,m)$, Stirling numbers of the second kind. 
We call this ``Poisson subordination by \ldots" --- whatever the type of $Z_1$ is. Consider $Z_1$ given, and $Z_2=e^v-1$. We have
$$e^{xZ_2(v)}=\sum\frac{v^n}{n!}\,p_n^{(2)}(x) \qquad \text{where}\qquad p_n^{(2)}(x)=\sum S(n,m)x^m$$
and by the above composition theorem we have, with \hfill\break $Z(v)=Z_1(\sigma(e^{\gamma v}-1))$,
$$p_n(x)=\gamma^n\sum S(n,m)\sigma^m\,p_{m}^{(1)}(x)\ .$$
\end{subsubsection}
\end{subsection}
\end{section}
\begin{section}{Inverse Pairs}	
Start with $V(z)$. We have $Y=xW(D)$ and polynomials $p_n(x)$ with generating function	
$$e^{vY}1 = e^{xZ(v)}=\sum_{n\ge0}\frac{v^n}{n!}\,p_n(x)$$
with operators acting on functions of $x$, in particular $z$ corresponds to $D=d/dx$.\\

Now consider the inverse $Z(v)$ as the canonical operator for its own system of polynomials in the variable $y$. Then form $Z'(v)$ and the
(dual) canonical operator
$$X=yZ'(\partial_y)^{-1}$$
Here we correspond $v$ to $\partial_y=d/dy$. 
Now operators act on functions of $y$ and we have the corresponding canonical polynomial sequence:
$$e^{zX}1 = e^{yV(z)}=\sum_{n\ge0}\frac{z^n}{n!}\,q_n(y)$$
with $Q$-matrix $\{q_{nk}\}$:
$$q_n(y)=\sum q_{nk}y^k$$
satisfying
\begin{align*}
X q_n&=q_{n+1}\cr
\mathcal{Z} q_n&=nq_{n-1}
\end{align*}
with $\mathcal{Z}=Z(d/dy)$.\\

Apply the composition theorem with $Z_1=Z$ and $Z_2=V$, $\sigma=\gamma=1$. 
We get
$$e^{xZ_1(Z_2(v))}=e^{xv}=\sum\frac{v^n}{n!}\,\sum\sum q_{nm}p_m(x)=\sum\frac{v^n}{n!}\,\sum\sum q_{nm}p_{mk}x^k$$
as expected, we have reciprocal relations, umbral composition,
$$\sum q_{nm}p_m(x)=x^n=q_n(Y)1$$ 
and the $P$ and $Q$ matrices inverse to each other:
$$\sum q_{nm}p_{mk}=\delta_{nk}$$
And dually,
$$\sum p_{nm}q_m(y)=y^n=p_n(X)1$$ 

We identify $z$ with $d/dx$ and $v$ with $d/dy$, with corresponding operators for capitalized variables.\\

We thus have:
\begin{subsection}{Commutation Relations}
The commutation relations are:
\begin{align*}
[z,x]&=I  \qquad \text{and} \qquad  [\mathcal{Z},X]=I\\
[v,y]&=I  \qquad \text{and} \qquad  [\mathcal{V},Y]=I
\end{align*}

These are four sets of \textsl{canonical pairs}.

The two reciprocal systems are \textsl{inverse pairs}.
\end{subsection}

Here is a summary illustrating the principal constructions, showing the duality between the systems of inverse pairs. 
From one side to the other, capitals are replaced by lower case and vice versa. \\

\begin{table}
\begin{align*}
z \sim \frac{d}{dx} &\qquad \mathcal{Z}\sim \frac{d}{dX} 
& v \sim \frac{d}{dy}& \qquad\mathcal{V}\sim \frac{d}{dY} \\
\mathstrut\\
&V(z) && Z(v)\\
&V'(z) && Z'(v)\\
\mathstrut\\
Y &= xV'(z)^{-1}=xZ'(\mathcal{V}) &X &= yZ'(v)^{-1}=yV'(\mathcal{Z})\\
x&=Y V'(z)=Y Z'(\mathcal{V})^{-1}  &y&=X Z'(v)=X V'(\mathcal{Z})^{-1}\\
\mathstrut\\
Y p_n&=p_{n+1}  &X q_n&=q_{n+1}\\
\mathcal{V} p_n&=np_{n-1} & \mathcal{Z} q_n&=nq_{n-1}\\
\mathstrut\\
p_n(X)1&=y^n & q_n(Y)1&=x^n\\
\mathstrut\\
e^{vY}1 &= e^{xZ(v)}=\sum_{n\ge0}\frac{v^n}{n!}\,p_n(x) & e^{zX}1 &= e^{yV(z)}=\sum_{n\ge0}\frac{z^n}{n!}\,q_n(y)
\end{align*}
\caption{\sc Generic system}
\end{table}
Note that the formulation $x=Y Z'(\mathcal{V})^{-1}$ gives the action of multiplication by $x$ on the polynomial sequence $\{p_n\}$. Similarly for
the dual system in terms of $y$ and the (inverse) sequence $\{q_n\}$.\\

We proceed to study some examples in detail. For each system, we find the main operators involved. Some of 
the $p$- and $q$-polynomials are presented, as the coefficients arising in these systems often have interesting
combinatorial significance (which, however, we will not discuss here).
\end{section}
\vfill\eject
\begin{section}{Poisson-Stirling system}
The name Poisson comes from the moment generating function for a Poisson distribution with parameter $y$:
$$e^{-y}\sum_{n\ge0} e^{zn}\,\frac{y^n}{n!}=e^{y(e^z-1)}$$
We have for canonical variables
\begin{align*}
&V(z)=e^z-1 && Z(v)=\log(1+v)\\
&V'(z)=e^z && Z'(v)=\frac{1}{1+v}\\
Y &= xe^{-z}=x(1+\mathcal{V})^{-1}&X &= y(1+v)=ye^{\mathcal{Z}}\\
x&=Y e^z=Y (1+\mathcal{V})&y&=X \frac{1}{1+v}=X e^{-\mathcal{Z}}
\end{align*}
The generating functions are
$$e^{vY}1 = (1+v)^x=\sum_{n\ge0}\frac{v^n}{n!}\,p_n(x)$$
and
$$e^{zX}1 = e^{y(e^z-1)}=\sum_{n\ge0}\frac{z^n}{n!}\,q_n(y)\ .$$
\begin{proposition}\rm We have
\begin{align*}
p_n(x)&=x(x-1)\cdots(x-(n-1))=\sum s(n,k)\,x^k\\
q_n(y)&=\sum S(n,k)\,y^k=\mathcal{T}_n(y)
\end{align*}
with $s(n,k)$ and $S(n,k)$ denoting Stirling numbers of the first and second kinds respectively, cf. \cite{DLMF}.\\
The polynomials $q_n$ are \textsl{Touchard polynomials}, which we denote by $\mathcal{T}$.
\end{proposition}
\begin{proof} For the $q$-polynomials, use the formula
$$\frac{(e^z-1)^m}{m!}=\sum_n S(n,m) \frac{z^n}{n!}\ .$$
\end{proof}

Selection of each:

\begin{align*}
p_0&=1\\
p_1&=x\\
p_2&=x^{2} - x\\
p_3&=x^{3} - 3 x^{2} + 2 x\\
p_4&=x^{4} - 6 x^{3} + 11 x^{2} - 6 x\\
p_5&=x^{5} - 10 x^{4} + 35 x^{3} - 50 x^{2} + 24 x\\
p_6&=x^{6} - 15 x^{5} + 85 x^{4} - 225 x^{3} + 274 x^{2} - 120 x\\
p_7&=x^{7} - 21 x^{6} + 175 x^{5} - 735 x^{4} + 1624 x^{3} - 1764 x^{2} + 720 x\\
p_8&=x^{8} - 28 x^{7} + 322 x^{6} - 1960 x^{5} + 6769 x^{4} - 13132 x^{3} + 13068 x^{2} - 5040 x\\
\end{align*}

\begin{align*}
q_0&=1\\
q_1&=y\\
q_2&=y^{2} + y\\
q_3&=y^{3} + 3 y^{2} + y\\
q_4&=y^{4} + 6 y^{3} + 7 y^{2} + y\\
q_5&=y^{5} + 10 y^{4} + 25 y^{3} + 15 y^{2} + y\\
q_6&=y^{6} + 15 y^{5} + 65 y^{4} + 90 y^{3} + 31 y^{2} + y\\
q_7&=y^{7} + 21 y^{6} + 140 y^{5} + 350 y^{4} + 301 y^{3} + 63 y^{2} + y\\
q_8&=y^{8} + 28 y^{7} + 266 y^{6} + 1050 y^{5} + 1701 y^{4} + 966 y^{3} + 127 y^{2} + y\\
\end{align*}

\begin{subsection}{Recurrences}
We have the standard actions $Yp_n=p_{n+1}$ and $\mathcal{V}p_n=np_{n-1}$ and correspondingly for
$X$ and $\mathcal{Z}$ acting on the $q$-polynomials. For example, $X=y(1+v)=y+yv$ gives the relation
$$q_{n+1}=yq_n+y\frac{dq_n}{dy}\ .$$

On the polynomials $\{p_n\}$, the formula $x=Y(1+\mathcal{V})=Y+Y\mathcal{V}$ reads
$$xp_n=p_{n+1}+np_n$$
so
$$p_{n+1}=(x-n)p_n$$
verifying the formula found above for $p_n$. While the relation $y=Xe^{-\mathcal{Z}}$ gives a recurrence as the expansion
$$yq_n=q_{n+1}-nq_n+\binom{n}{2}q_{n-1}-\cdots$$
We derive this using the fact that $e^{-\mathcal{Z}}$ acts as a shift operator on functions of $X$:
$$e^{-\mathcal{Z}}f(X)1=f(X-1)1$$
so
$$yq_n=X(X-1)^n1=\sum_k \binom{n}{k}(-1)^kq_{n+1-k}$$
as stated above.

And we note that the semigroup properties holding in general for the polynomials' coefficients give convolution relations for the Stirling numbers
of the first and second kinds.
\end{subsection}
\end{section}
\begin{section}{Laguerre system}
We have for canonical variables
\begin{align*}
&V(z)=\frac{z}{1-z} && Z(v)=\frac{v}{1+v}\\
&V'(z)=\frac{1}{(1-z)^2} && Z'(v)=\frac{1}{(1+v)^2}\\
Y &= x(1-z)^2=x(1+\mathcal{V})^{-2}&X &= y(1+v)^2=y(1-\mathcal{Z})^{-2}\\\\
x&=Y(1-z)^{-2} =Y (1+\mathcal{V})^2    &y&=X (1+v)^{-2}=X (1-\mathcal{Z})^2
\end{align*}
The generating functions are
$$e^{vY}1 = e^{\textstyle x\,\frac{\textstyle v}{\textstyle 1+v}}=\sum_{n\ge0}\frac{v^n}{n!}\,p_n(x)$$
and
$$e^{zX}1 =e^{\textstyle y\,\frac{\textstyle z}{\textstyle 1-z}}=\sum_{n\ge0}\frac{z^n}{n!}\,q_n(y)$$
We use the notation from online \texttt{dlmf.nist.gov}, \cite[18.5.12]{DLMF}.
\begin{proposition}\rm We have
\begin{align*}
p_n(x)&=(-1)^n\,n!\,L_n^{(-1)}(x)	\\
q_n(y)&=\sum_{k=1}^n\binom{n}{k}\frac{\Gamma(n)}{\Gamma(k)}\,y^k=n!\,L_n^{(-1)}(-y)
\end{align*}

The polynomials $q_n$ are special Laguerre polynomials.
\end{proposition}
Selection of each:

\begin{align*}
p_{0} &= 1\\
p_{1} &= x\\
p_{2} &= x^{2} - 2 x\\
p_{3} &= x^{3} - 6 x^{2} + 6 x\\
p_{4} &= x^{4} - 12 x^{3} + 36 x^{2} - 24 x\\
p_{5} &= x^{5} - 20 x^{4} + 120 x^{3} - 240 x^{2} + 120 x\\
p_{6} &= x^{6} - 30 x^{5} + 300 x^{4} - 1200 x^{3} + 1800 x^{2} - 720 x\\
p_{7} &= x^{7} - 42 x^{6} + 630 x^{5} - 4200 x^{4} + 12600 x^{3} - 15120 x^{2} + 5040 x\\
p_{8} &= x^{8} - 56 x^{7} + 1176 x^{6} - 11760 x^{5} + 58800 x^{4} - 141120 x^{3} + 141120 x^{2} - 40320 x\\
\end{align*}

\begin{align*}
q_{0} &= 1\\
q_{1} &= y\\
q_{2} &= y^{2} + 2 y\\
q_{3} &= y^{3} + 6 y^{2} + 6 y\\
q_{4} &= y^{4} + 12 y^{3} + 36 y^{2} + 24 y\\
q_{5} &= y^{5} + 20 y^{4} + 120 y^{3} + 240 y^{2} + 120 y\\
q_{6} &= y^{6} + 30 y^{5} + 300 y^{4} + 1200 y^{3} + 1800 y^{2} + 720 y\\
q_{7} &= y^{7} + 42 y^{6} + 630 y^{5} + 4200 y^{4} + 12600 y^{3} + 15120 y^{2} + 5040 y\\
q_{8} &= y^{8} + 56 y^{7} + 1176 y^{6} + 11760 y^{5} + 58800 y^{4} + 141120 y^{3} + 141120 y^{2} + 40320 y\\ 
\end{align*}

\begin{subsection}{Recurrences}
From $x=Y (1+\mathcal{V})^2=Y+2Y\mathcal{V}+Y\mathcal{V}^2$ we get
$$xp_n=p_{n+1}+2np_n+n(n-1)p_{n-1}$$
or
$$p_{n+1}=(x-2n)p_n-n(n-1)p_{n-1}$$
Similarly, $y=X-2X\mathcal{Z}+X\mathcal{Z}^2$ gives
$$q_{n+1}=(y+2n)q_n-n(n-1)q_{n-1}$$
while $Y=x(1-z)^2$ gives the differential recurrence
$$p_{n+1}=x(p_n-2p_n'+p_n'')\ .$$
\end{subsection}

\end{section}

\begin{section}{Hermite-Bessel system}
Here we are looking at polynomials related to the Gaussian distribution. We have for canonical variables
\begin{align*}
&V(z)=z-z^2/2 && Z(v)=1-\sqrt{1-2v}\\
&V'(z)= 1-z && Z'(v)=\frac{1}{\sqrt{1-2v}}\\
Y &= x(1-z)^{-1}=x(1-2\mathcal{V})^{-1/2}       &X &= y\sqrt{1-2v}=y(1-\mathcal{Z})\\
x&=Y (1-z)=Y (1-2\mathcal{V})^{1/2}    &y&=X \frac{1}{\sqrt{1-2v}}=X \frac{1}{1-\mathcal{Z}}
\end{align*}
The generating functions are
$$e^{vY}1 = e^{x(1-\sqrt{1-2v})}=\sum_{n\ge0}\frac{v^n}{n!}\,p_n(x)$$
and
$$e^{zX}1 =e^{y(z-z^2/2)}=\sum_{n\ge0}\frac{z^n}{n!}\,q_n(y)$$

For the modified Hermite polynomials, see \cite[18.12.16]{DLMF}. For Bessel polynomials, cf. \cite[18.34.2]{DLMF}.
\begin{proposition} \rm We have
\begin{align*}
p_n(x)&=x\theta_{n-1}(x)=\sum \frac{\Gamma(n+k)}{\Gamma(n-k)\,k!} \frac{x^{n-k}}{2^k}\\
q_n(y)&=He_n(y,y)=\sum\binom{n}{2k}\frac{(2k)!}{2^k\,k!}	(-1)^ky^{n-k}
\end{align*}
where $\theta_n$ are Bessel polynomials and $He_n(x,t)$ are probabilists' Hermite polynomials
(orthogonal with respect to a Gaussian distribution with mean zero and variance $t$) with generating function
\begin{equation}\label{eq:He}
e^{xz-z^2t/2}=\sum_{n\ge0}\frac{z^n}{n!}\,He_n(x,t)
\end{equation}
\end{proposition}
For Bessel polynomials, the formulation 
$$ x\theta_{n-1}(x)=\sum_{k=0}^{n-1} \binom{n+k-1}{2k}\frac{(2k)!}{2^k\,k!}x^{n-k}$$
shows clearly the connection with moments of the Gaussian distribution or, in this case, one should say moments of the $\chi^2-$distribution.\\

Selection of each:

\begin{align*}
p_0&=1\\
p_1&=x\\
p_2&=x^{2} + x\\
p_3&=x^{3} + 3 x^{2} + 3 x\\
p_4&=x^{4} + 6 x^{3} + 15 x^{2} + 15 x\\
p_5&=x^{5} + 10 x^{4} + 45 x^{3} + 105 x^{2} + 105 x\\
p_6&=x^{6} + 15 x^{5} + 105 x^{4} + 420 x^{3} + 945 x^{2} + 945 x\\
p_7&=x^{7} + 21 x^{6} + 210 x^{5} + 1260 x^{4} + 4725 x^{3} + 10395 x^{2} + 10395 x\\
p_8&=x^{8} + 28 x^{7} + 378 x^{6} + 3150 x^{5} + 17325 x^{4} + 62370 x^{3} + 135135 x^{2} + 135135 x\\
\end{align*}
\begin{align*}
q_0&=1\\
q_1&=y\\
q_2&=y^{2} - y\\
q_3&=y^{3} - 3 y^{2}\\
q_4&=y^{4} - 6 y^{3} + 3 y^{2}\\
q_5&=y^{5} - 10 y^{4} + 15 y^{3}\\
q_6&=y^{6} - 15 y^{5} + 45 y^{4} - 15 y^{3}\\
q_7&=y^{7} - 21 y^{6} + 105 y^{5} - 105 y^{4}\\
q_8&=y^{8} - 28 y^{7} + 210 y^{6} - 420 y^{5} + 105 y^{4}\\
\end{align*}
Notice the diagonals along the Hermite polynomials yield the coefficients for the Bessel polynomials up to sign.\bigskip

\begin{subsection}{Recurrences}
With $x=Y (1-2\mathcal{V})^{1/2} $ we calculate
\begin{align*}
x^2&=Y(1-2\mathcal{V})^{1/2}Y(1-2\mathcal{V})^{1/2}\\
&=Y^2(1-2\mathcal{V})+Y(\frac{d}{d\mathcal{V}}(1-2\mathcal{V})^{1/2})(1-2\mathcal{V})^{1/2}\\
&=Y^2(1-2\mathcal{V})-Y
\end{align*}
Hence,
$$x^2p_{n}=p_{n+2}-(2n+1)p_{n+1}$$
or, shifting the index $n$ back by one, and rearranging,
$$p_{n+1}=(2n-1)p_n+x^2p_{n-1}$$
a recurrence formula for Bessel polynomials. For Hermite polynomials, we have
$X=y-y\mathcal{Z}$ yielding
$$q_{n+1}=yq_n-ynq_{n-1}=y(q_n-nq_{n-1})\ .$$
\end{subsection}
\end{section}
\begin{section}{Arcsinh system}

\begin{align*}
&V(z)=\arcsinh z && Z(v)=\sinh v\\
&V'(z)= \frac{1}{\sqrt{1+z^2}} && Z'(v)= \cosh v\\
Y &= x\sqrt{1+z^2}=x\cosh \mathcal{V}      &X &=  y\,\sech v=y\frac{1}{\sqrt{1+\mathcal{Z}^2}}\\
x&=Y \frac{1}{\sqrt{1+z^2}}=Y \sech \mathcal{V}   &y&=X \cosh v=X \sqrt{1+\mathcal{Z}^2}
\end{align*}

The generating functions are
$$e^{vY}1 = e^{x\sinh v}=\sum_{n\ge0}\frac{v^n}{n!}\,p_n(x)$$
and
$$e^{zX}1 = e^{y\,\arcsinh z}=\sum_{n\ge0}\frac{z^n}{n!}\,q_n(y)$$

\begin{proposition} \rm We have, denoting the Touchard polynomials $\mathcal{T}_n(x)$,	
\begin{align*}
p_n(x)&=\sum_j \binom{n}{j}(-1)^j\mathcal{T}_{n-j}(x/2)\mathcal{T}_j(-x/2)\\
q_n(y)&=\begin{cases}
y\,\prod\limits_{k=0}^{(n-3)/2} (y^2-(2k+1)^2)\,,& n \text{ odd}\\
\mathstrut\\
\prod\limits_{k=0}^{(n-2)/2} (y^2-(2k)^2)\,,& n \text{ even}\\
\end{cases}
\end{align*}
\end{proposition}
\begin{proof} For $\{p_n\}$, write $\sinh v=(e^v-e^{-v})/2$ and use the Poisson-Stirling $q$-polynomials:
\begin{align*}
e^{x\sinh v}&=e^{x(e^v-1)/2}e^{-x(e^{-v}-1)/2}\\
&=\left(\sum_n\frac{v^n}{n!}\mathcal{T}_n(x/2)\right)\left(\sum_n\frac{(-1)^nv^n}{n!}\mathcal{T}_n(-x/2)\right)
\end{align*}
which yields the result upon rearrangement. The formula for the $q_n$ will be seen directly from the recurrence derived below.
\end{proof}

Selection of each:

\begin{align*}
p_0&=1\\
p_1&=x\\
p_2&=x^{2}\\
p_3&=x^{3} + x\\
p_4&=x^{4} + 4 x^{2}\\
p_5&=x^{5} + 10 x^{3} + x\\
p_6&=x^{6} + 20 x^{4} + 16 x^{2}\\
p_7&=x^{7} + 35 x^{5} + 91 x^{3} + x\\
p_8&=x^{8} + 56 x^{6} + 336 x^{4} + 64 x^{2}\\
\end{align*}
\begin{align*}
q_0&=1\\
q_1&=y\\
q_2&=y^{2}\\
q_3&=y^{3} - y\\
q_4&=y^{4} - 4 y^{2}\\
q_5&=y^{5} - 10 y^{3} + 9 y\\
q_6&=y^{6} - 20 y^{4} + 64 y^{2}\\
q_7&=y^{7} - 35 y^{5} + 259 y^{3} - 225 y\\
q_8&=y^{8} - 56 y^{6} + 784 y^{4} - 2304 y^{2}\\
\end{align*}

\begin{subsection}{Recurrences. Bessel operator.}
Here we have the interesting appearance of the Bessel operator. Effectively, $Y$ is the square root of the Bessel
operator $(xD)^2+x^2$. To see this, calculate, with $Y=x\sqrt{1+z^2}$, $z=D=d/dx$,
\begin{align*}
Y^2&=x\sqrt{1+z^2}\,x\sqrt{1+z^2}\\
&=x^2(1+z^2)+x\left(\frac{d}{dz}(\sqrt{1+z^2})\right)\sqrt{1+z^2}\\
&=x^2(1+z^2)+xz\\
&=x^2+(xz)^2\ .
\end{align*}
Thus, the polynomials $p_n$ are formed by repeated application of the Bessel operator. The differential recurrence has the form:
$$p_{n+2}=x^2p_n+\left(x\frac{d}{dx}\right)^2p_n$$

For the $q$-sequence, use the dual relation $y=X\sqrt{1+\mathcal{Z}^2}$ to get similarly
$$y^2=X^2+(X\mathcal{Z})^2$$
That is,
$$y^2q_n=q_{n+2}+n^2q_n$$
or
$$q_{n+2}=(y^2-n^2)q_n$$
and hence the formulas quoted above.
\end{subsection}
\end{section}
\section{Tanh system}

\begin{align*}
&V(z)=\tanh z && Z(v)=\arctanh v\\
&V'(z)={\rm sech}^2z && Z'(v)= \frac{1}{1-v^2}\\
Y &=  x\cosh^2z=x\,\frac{1}{1-\mathcal{V}^2}      &X &=  y(1-v^2)=y\,{\rm sech}^2\mathcal{Z}\\
x&=Y\,{\rm sech}^2z=Y(1-\mathcal{V}^2)   &y&=X\frac{1}{1-v^2}=X\cosh^2\mathcal{Z}
\end{align*}
and the generating functions are
$$e^{vY}1 = e^{x\,\arctanh v}=\sum_{n\ge0}\frac{v^n}{n!}\,p_n(x)$$
$$e^{zX}1 = e^{y\tanh z}=\sum_{n\ge0}\frac{z^n}{n!}\,q_n(y)$$

\begin{proposition} \rm We have the ``zero-step" Krawtchouk polynomials
$$p_n(x)=\sum_k \binom{n}{k}(-1)^k\left(\frac{x}{2}\right)^{(n-k)}\left(-\frac{x}{2}\right)^{(k)}=(-1)^n\,K_n(x/2;1/2,0)$$
cf. \cite[18.23.3]{DLMF}, dropping the factor involving $N$.\\
We have $a^{(n)}=a(a-1)\cdots(a-n+1)$ denoting falling factorial (factorial power)\\

and for the $q$-polynomials we have Poisson subordination by Laguerre:
$$q_n(y)=\sum_m S(n,m)\,2^{n-m}\,p_m^{\rm Laguerre}(y)\ .$$
\end{proposition}
\begin{proof}Recall
$$\tanh z=\frac{e^{2z}-1}{e^{2z}-1} \qquad \text{and}\qquad \arctanh v=\frac{1}{2}\,\log\bigl(\frac{1+v}{1-v}\bigr)$$
So for $p$-polynomials
$$\left(\frac{1+v}{1-v}\right)^{x/2}=\sum_n\frac{v^n}{n!}\,p_n(x)\ .$$
Expanding by the binomial theorem and combining terms yields the formula.\\

For the $q$-polynomials, refer to Section \S\ref{sec:PS}. We verify Poisson subordination with $Z_1(v)=v/(1+v)$, $\sigma=1/2$, $\gamma=2$.
Here we want the expansion in powers of $z$. We have $v=(1/2)(e^{2z}-1)$ and, exchanging $Z$ and $V$:
$$V=\frac{1}{2}\,\frac{e^{2z}-1}{1+(1/2)(e^{2z}-1)}=\frac{e^{2z}-1}{e^{2z}+1}=\tanh z$$
thus the result for the $q$-polynomials.
\end{proof}

Selection of each:

\begin{align*}
p_0&=1\\
p_1&=x\\
p_2&=x^{2}\\
p_3&=x^{3} + 2 x\\
p_4&=x^{4} + 8 x^{2}\\
p_5&=x^{5} + 20 x^{3} + 24 x\\
p_6&=x^{6} + 40 x^{4} + 184 x^{2}\\
p_7&=x^{7} + 70 x^{5} + 784 x^{3} + 720 x\\
p_8&=x^{8} + 112 x^{6} + 2464 x^{4} + 8448 x^{2}\\
\end{align*}
\begin{align*}
q_0&=1\\
q_1&=y\\
q_2&=y^{2}\\
q_3&=y^{3} - 2 y\\
q_4&=y^{4} - 8 y^{2}\\
q_5&=y^{5} - 20 y^{3} + 16 y\\
q_6&=y^{6} - 40 y^{4} + 136 y^{2}\\
q_7&=y^{7} - 70 y^{5} + 616 y^{3} - 272 y\\
q_8&=y^{8} - 112 y^{6} + 2016 y^{4} - 3968 y^{2}\\
\end{align*}

\begin{subsection}{Recurrences}
With $x=Y (1-\mathcal{V}^2) $ we have directly
$$xp_{n}=p_{n+1}-n(n-1)p_{n-1}$$
and the recurrence
$$p_{n+1}=xp_n+n(n-1)p_{n-1}\ .$$

Duallly, we have $X=y(1-v^2)$ yielding the differential recurrence
$$q_{n+1}=yq_n-yq_{n}''=y(q_n-q_n'')\ .$$

From $y=X\cosh^2\mathcal{Z}$, writing $\cosh^2\mathcal{Z}=\bigl(e^{2\mathcal{Z}}+e^{-2\mathcal{Z}}+2\bigr)/4$, we have
\begin{align*}
yq_n&=\frac{X}{4}\bigl[(X+2)^n+(X-2)^n+2X^n\bigr]1\\
&=\frac{X}{4}\left[\sum_k\binom{n}{k} X^{n-k}(2^k+(-2)^{k}) \right]1+\frac{1}{2}q_{n+1}\\
&=\frac{1}{4}\sum_j\binom{n}{2j}X^{n+1-2j}2^{2j+1}1+\frac{1}{2}q_{n+1}
\end{align*}
or, taking out the term $j=0$,
$$yq_n=\sum_{j\ge1}\binom{n}{2j}2^{2j-1}q_{n+1-2j}+q_{n+1}$$
and finally,
$$q_{n+1}=yq_n-\sum_{j\ge1}\binom{n}{2j}2^{2j-1}q_{n+1-2j}\ .$$

\end{subsection}

\begin{section}{Gegenbauer system}

\begin{align*}
&V(z)=-\log(1-2\alpha z+z^2) && Z(v)=\alpha-\sqrt{\alpha^2 + e^{- v}-1}\\
&V'(z)=\frac{2\alpha-2z}{1-2\alpha z+z^2} && Z'(v)=\frac{e^{-v}}{2\sqrt{\alpha^2 + e^{- v}-1}}\\
\mathstrut\\
Y &= x\frac{1-2\alpha z+z^2}{2\alpha-2z} =x\frac{e^{-\mathcal{V}}}{2\sqrt{\alpha^2 + e^{- \mathcal{V}}-1}}
 &X &= 2ye^{v}\sqrt{\alpha^2 + e^{- v}-1}=y\frac{2\alpha-2\mathcal{Z}}{1-2\alpha \mathcal{Z}+\mathcal{Z}^2}\\
\mathstrut\\
x&=Y\frac{2\alpha-2z}{1-2\alpha z+z^2} =2Ye^{\mathcal{V}}\sqrt{\alpha^2 + e^{- \mathcal{V}}-1}   
&y&=X\frac{e^{-v}}{2\sqrt{\alpha^2 + e^{- v}-1}} =X\frac{1-2\alpha\mathcal{Z}+\mathcal{Z}^2}{2\alpha-2\mathcal{Z}}
\end{align*}
and the generating functions are
$$e^{vY}1 = e^{x\,Z(v)}=\sum_{n\ge0}\frac{v^n}{n!}\,p_n(x)$$
$$e^{zX}1 = e^{yV(z)}=(1-2\alpha z+z^2)^{-y}=\sum_{n\ge0}\frac{z^n}{n!}\,q_n(y)$$

\begin{proposition}\rm
$$p_n(x)=\sum_k S(n,k)(-1)^{n-k}(2\alpha^2)^{-k} p_k^{\rm Bessel}(\alpha x)$$
where $p_k^{\rm Bessel}$ are the Bessel polynomials $x\theta_{n-1}(x)$ as in the Hermite-Bessel system.
That is, these arise as Poisson subordination by Bessel, with scaling $x$ as well. And
$$q_n(x)=n!\,C_n^{(x)}(\alpha)$$
Gegenbauer polynomials in the variable $\alpha$ with parameter $x$. 
\end{proposition}
\begin{proof}
For the form of the $p$-polynomials, we verify Poisson subordination with $Z_1(v)=1-\sqrt{1-2v}$, $\sigma=-(2\alpha^2)^{-1}$, 
$\gamma=-1$. So we have, scaling by $\alpha$,
$$Z=\alpha\left(1-\sqrt{1-2\frac{-1}{2\alpha^2}(e^{-v}-1)}\ \right)$$
in agreement with $Z$ as given above.
\end{proof}

Selection of each:

\begin{align*}
p_0&=1\\
p_1&=\frac{x}{2 \alpha}\\
p_2&=x \left(- \frac{1}{2 \alpha} + \frac{1}{4 \alpha^{3}}\right) + \frac{x^{2}}{4 \alpha^{2}}\\
p_3&=x^{2} \left(- \frac{3}{4 \alpha^{2}} + \frac{3}{8 \alpha^{4}}\right) + x \left(\frac{1}{2 \alpha} - \frac{3}{4 \alpha^{3}} + \frac{3}{8 \alpha^{5}}\right) + \frac{x^{3}}{8 \alpha^{3}}\\
p_4&=x^{3} \left(- \frac{3}{4 \alpha^{3}} + \frac{3}{8 \alpha^{5}}\right) + x^{2} \left(\frac{7}{4 \alpha^{2}} - \frac{9}{4 \alpha^{4}} + \frac{15}{16 \alpha^{6}}\right) + x \left(- \frac{1}{2 \alpha} + \frac{7}{4 \alpha^{3}} - \frac{9}{4 \alpha^{5}} + \frac{15}{16 \alpha^{7}}\right)\\&\qquad + \frac{x^{4}}{16 \alpha^{4}}\\
p_5&=x^{4} \left(- \frac{5}{8 \alpha^{4}} + \frac{5}{16 \alpha^{6}}\right) + x^{3} \left(\frac{25}{8 \alpha^{3}} - \frac{15}{4 \alpha^{5}} + \frac{45}{32 \alpha^{7}}\right) + x^{2} \left(- \frac{15}{4 \alpha^{2}} + \frac{75}{8 \alpha^{4}} - \frac{75}{8 \alpha^{6}} + \frac{105}{32 \alpha^{8}}\right) \\
&\qquad+ x \left(\frac{1}{2 \alpha} - \frac{15}{4 \alpha^{3}} + \frac{75}{8 \alpha^{5}} - \frac{75}{8 \alpha^{7}} + \frac{105}{32 \alpha^{9}}\right) + \frac{x^{5}}{32 \alpha^{5}}
\end{align*}
\begin{align*}
q_0&=1\\
q_1&=2 \alpha y\\
q_2&=4 \alpha^{2} y^{2} + y \left(4 \alpha^{2} - 2\right)\\
q_3&=8 \alpha^{3} y^{3} + y^{2} \left(24 \alpha^{3} - 12 \alpha\right) + y \left(16 \alpha^{3} - 12 \alpha\right)\\
q_4&=16 \alpha^{4} y^{4} + y^{3} \left(96 \alpha^{4} - 48 \alpha^{2}\right) + y^{2} \left(176 \alpha^{4} - 144 \alpha^{2} + 12\right) + y \left(96 \alpha^{4} - 96 \alpha^{2} + 12\right)\\
q_5&=32 \alpha^{5} y^{5} + y^{4} \left(320 \alpha^{5} - 160 \alpha^{3}\right) + y^{3} \left(1120 \alpha^{5} - 960 \alpha^{3} + 120 \alpha\right)\\
&\qquad + y^{2} \left(1600 \alpha^{5} - 1760 \alpha^{3} + 360 \alpha\right) + y \left(768 \alpha^{5} - 960 \alpha^{3} + 240 \alpha\right)
\end{align*}

\begin{subsection}{Recurrence}
From 
$$y=X \frac{1-2\alpha \mathcal{Z}+\mathcal{Z}^2}{2\alpha-2\mathcal{Z}}$$
we have
$$y(2\alpha-2\mathcal{Z})=X(1-2\alpha \mathcal{Z}+\mathcal{Z}^2)$$
Applying this relation to $q_n$ yields
$$y(2\alpha q_n-2nq_{n-1})=q_{n+1}-2\alpha n q_n+n(n-1)q_{n-1}$$
and rearranging, we have
$$q_{n+1}=2\alpha(y+n)q_n-n(2y+n-1)q_{n-1}\ .$$
\end{subsection}
\end{section}

\begin{section}{Multivariate Case}
We indicate how the multivariate case works. For $N>0$, $V(z)=(V_1(z),\ldots,V_N(z))$, $z=(z_1,\ldots,z_N)$, with $Z(V(z))=z$, the functional
inverse. $V(0)=0\in \mathbb{C}^N$, holomorphic in a neighborhood of the origin. Now $V'(z)$ is the Jacobian 
$\displaystyle\bigl(\frac{\partial V_i}{\partial z_j}\bigr)$, with inverse matrix $W(z)$. The canonical variables are
$$Y_j=x_\lambda W_{\lambda j}(D)$$
with repeated Greek indices summed. One checks the commutation relations
\begin{align*}
[V_i(D),Y_j]&=\delta_{ij}\,I\\
[Y_i,Y_j]=0
\end{align*}
The canonical polynomials are
$$p_n(x)=Y_1^{n_1}\cdots Y_N^{n_N}1$$
with multi-index $n=(n_1,\ldots,n_N)$. The generating function is
$$e^{x\cdot Z(v)}=\sum_{n}\frac{v^n}{n!}\,p_n(x)$$
with $x\cdot Z(v)=\sum_j x_j Z_j(v)$, dot product, and the usual conventions for multi-indices.\\

We illustrate with a multivariate version of the Hermite-Bessel system.\\

\begin{example}
Let $V_j(z)=z_j-\frac{1}{2}\sum z_i^2$. Then we find $Z_j(v)=v_j+S(v)$ where
$$S(v)=\frac{1}{N}\left(1-\sum v_j-\sqrt{(1-\sum v_j)^2-N\sum v_j^2} \right)$$
The Jacobian is
$$V'=I-Z$$
where $Z$ is the rank one matrix with every row equal to $(z_1,\ldots,z_N)$. And
$$W=(V')^{-1}=I+\frac{1}{1-\sum z_i}\,Z$$
so
$$Y_j=x_j+({\textstyle \sum x_i})\frac{D_j}{1-\sum D_i}\ .$$

The $p$-polynomials are a family of multivariate Bessel polynomials. We can find the $q$-polynomials expressed as products of
Hermite polynomials. Note
\begin{align*}
y\cdot V(z)&=\sum y_j\left(z_j-\frac{1}{2}\,\sum z_i^2\right)\\
&=\sum y_jz_j-\frac{1}{2}\bigl({\textstyle\sum y_i}\bigr)\sum z_j^2\\
&=\sum\left[y_jz_j-\frac{1}{2}\bigl({\textstyle\sum y_i}\bigr)z_j^2\right]
\end{align*}
So
\begin{align*}
e^{y\cdot V}&=\prod_j \exp\left(y_jz_j-\frac{1}{2}({\textstyle\sum y_i}) z_j^2\right)\\
&=\prod_j\sum_{n_j} \frac{z_j^{n_j}}{n_j!}He_{n_j}(y_j,{\textstyle\sum y_i})\\
&=\sum_n\frac{z^n}{n!} \prod_{n=(n_1,\ldots,n_N)}He_{n_j}(y_j,{\textstyle\sum y_i})
\end{align*}
cf. equation \eqref{eq:He}.
\end{example}
\begin{subsection}{Linear maps and symmetric tensor powers}
An interesting special case arises when $V$ is a linear map. Let $A$ be a given $N\times N$ matrix. 
Form 
$$V(z)=(A^{-1})^Tz$$
with $T$ denoting transpose. We have $V'=(A^{-1})^T$ and $W=A^T$. So
$$Y_j=x_\lambda W_{\lambda j}=x_\lambda A_{j\lambda}=(Ax)_j$$
i.e. $Y=Ax$. \\

And $Z(v)=A^{T}v$, with $Z'=A^T$ and $X=A^{-1}y$. The generating function is
$$e^{x\cdot A^Tv}=e^{Ax\cdot v}=\sum_n\frac{v^n}{n!}\,\sum_m {\bar A}_{nm}x^m$$
That is,
$$Y^n1=(Ax)^n=\sum {\bar A}_{nm} x^m$$
For given homogeneous degree $d$, the map $A\to \bar A$ is a multiplicative map taking
the space of homogeneous polynomials of degree $d$ to itself. For a matrix group,
for each $d$, this gives a group representation, the action of the group on
polynomials, alternatively, action on the space of symmetrized tensors (of the underlying vector space), cf. \cite{Sp}.\\
\begin{subsubsection}{Homorphism property}
The composition theorem here becomes the property that the map $A\to \bar A$ is a homomorphism. Take two matrices $A$ and $B$ and form
$Z_1=B^Tv$, $Z_2=A^Tv$. Then
$$Z_1(Z_2(v))=B^TA^Tv=(AB)^Tv$$
And the exponential
$$e^{xZ_1(Z_2(v))}=\sum\frac{v^n}{n!}\sum {\bar A}_{nk}{\bar B}_{km} x^m=\sum\frac{v^n}{n!}\sum {\overline {AB}}_{nm} x^m$$
as expected.
\end{subsubsection}
\begin{subsubsection}{Lie algebra map}
We can also study the induced map at the Lie algebra level. Namely, define $\Gamma(A)$ to be the matrix satisfying
$$\overline{\exp(tA)}=\exp(t\Gamma(A))$$
As the generator of the image of the one-parameter group generated by $A$ we have
$$\Gamma(A)=\frac{d}{dt}\biggm|_{t=0}\overline{\exp(tA)}\ .$$
\begin{proposition}  The generating function for the matrices $\Gamma(A)$ is given by
$$(Ax\cdot v)e^{x\cdot v}=\sum_n\frac{v^n}{n!}\,\sum\Gamma(A)_{nm}x^m$$
\end{proposition}
\begin{proof} Start with
$$\exp\bigl((e^{tA}x)\cdot v\bigr)=\sum\frac{v^n}{n!}\sum (\overline{e^{tA}})_{nm}x^m$$
Now differentiate with respect to $t$ and set $t=0$. We get
$$\bigl((Ae^{tA}x)\cdot v\bigr)\exp((e^{tA}x)\cdot v)\bigm|_{t=0}=(Ax\cdot v)e^{x\cdot v}=\sum\frac{v^n}{n!}\sum\Gamma(A)_{nm}x^m$$
as stated.
\end{proof}
This map plays an important part in certain multivariate orthogonal systems, see, e.g., \cite{F}.
\end{subsubsection}
\end{subsection}
\end{section}

\begin{section}{Concluding remarks}
One can work in the somewhat more general setting of Sheffer sequences. From our perspective a
natural way to approach these systems is as canonical systems evolving under the action of a one-parameter group.
Let $H(D)$ generate a one-parameter group acting on the polynomials $\{p_n\}$. We have
$$e^{tH(D)}e^{xZ(v)}=e^{xZ(v)+tH(Z(v))}=\sum\frac{v^n}{n!}\bigl(e^{tH(D)}p_n(x)\bigr)=\sum\frac{v^n}{n!}p_n(x,t)\ .$$
The relation
$$e^{tH(D)}xe^{-tH(D)}=x+tH'(D)$$ yields
the raising operator for the new system
$$\mathcal{R}=e^{tH(D)}Ye^{-tH(D)}=(x+tH'(D))W(D)=Y+tH'(D)W(D)$$
i.e. $\mathcal{R}p_n(x,t)=p_{n+1}(x,t)$ with $\mathcal{V}=V(D)$ remaining as the lowering operator, $[\mathcal{V},\mathcal{R}]=I$.
These systems play a principal r\^ole in \cite{R,SGWW}.
\end{section}

\end{document}